\documentclass[reqno]{amsart}
\usepackage{amsfonts}
\usepackage{graphicx}
\usepackage{amscd}
\usepackage{amsmath}
\usepackage{amssymb}
\usepackage[latin1]{inputenc}
\usepackage{amssymb,amsmath}
\usepackage{amsfonts}
\usepackage{graphicx,color}
\usepackage{amsthm}
\usepackage{graphicx}
\usepackage{fancyhdr}
\usepackage{hyperref}
\theoremstyle{plain}

\newtheorem{theorem}{\bf Theorem}[section]
\newtheorem{corollary}{\bf Corollary}[section]

\newtheorem{lemma}[theorem]{\bf Lemma}

\newtheorem{proposition*}{\bf Proposition}

\numberwithin{equation}{section}

\newcommand\dM{dv_M}
\newcommand\dB{dv_B}

\begin{document}

\title[Generic spectrum of warped products and G-manifolds]{Generic spectrum of warped products and G-manifolds}

\author{Marcus A.M. Marrocos$^1$}
\author{Jos\'e N.V. Gomes$^{2,3}$}

\address{$^1$CMCC-Universidade Federal do ABC, Av. dos Estados, 5001, 09210-580 Santo Andr\'e, S\~ao Paulo, Brazil.}
\address{$^2$Departamento de Matem\'atica, Universidade Federal do Amazonas, Av. General Rodrigo Oct\'avio, 6200, 69080-900 Manaus, Amazonas, Brazil.}
\address{$^3$Present Address: Department of Mathematics, Lehigh University, Christmas-Saucon Hall, 14 East Packer Avenue, Bethlehem, Pennsylvania, 18015-
3174, US.}

\email{$^1$marcus.marrocos@ufabc.edu.br}
\email{$^2$jnvgomes@pq.cnpq.br; jnvgomes@gmail.com}

\urladdr{$^1$http://cmcc.ufabc.edu.br}
\urladdr{$^2$https://ufam.edu.br; http://ppgm.ufam.edu.br}

\thanks{$^1$Partially supported by Grant 2016/10009-3, São Paulo Research Foundation (FAPESP)}
\thanks{$^2$Partially supported by Grant 202234/2017-7, Conselho Nacional de Desenvolvimento Cient\'ifico e Tecnol\'ogico (CNPq), of
the Ministry of Science, Technology and Innovation of Brazil}



\begin{abstract}
In this paper, we establish a kind of splitting theorem for the eigenvalues of a specific family of operators on the base of a warped product. As a consequence, we prove a density theorem for a set of warping functions that makes the spectrum of the Laplacian a warped-simple spectrum. This is then used to study the generic situation of the eigenvalues of the Laplacian on a class of compact G-manifolds. In particular, we give a partial answer to a question posed in 1990 by Steven Zelditch about the generic situation of multiplicity of the eigenvalues of the Laplacian on principal bundles.
\end{abstract}
\maketitle

\section{Introduction}

The analysis of the spectrum of the Laplacian on compact manifolds is a central topic in both mathematics and physics. Not only the literature about this subject is already very rich, but also it is not unlikely that the Laplacian may play a fundamental role in the understanding of countless physical facts. In this paper, we discuss an interesting case of the spectrum of the Laplacian which stems from work of Uhlenbeck~\cite{Uhlenbeck} from which we know that the set of $C^k$~Riemannian metrics on an $n$--dimensional compact smooth manifold $M^n$ whose Laplacian has simple spectrum is residual, for $2\leq k<\infty$. Moreover, the same result is also known for a family of elliptic operators $\Delta_b=\Delta+b$, for $b\in C^k(M)$.

Warped product appears in a natural manner in Riemannian geometry and their applications abound. They are source of examples and counter-examples of many classical geometric questions. In this work, we relate them to generic results such as of the Uhlenbeck. In short, we study warped products in order to obtain a class of $G$--manifolds (that contain principal bundles) so that is possible to describe their generic spectrum, cf. Corollaries~\ref{generic-G-eigenvalue} and \ref{question2}. By a $G$--manifold we mean a smooth manifold $M$ endowed with a smooth effective action by a Lie group $G$. In particular, we give a partial answer to a question posed in 1990 by Zelditch~\cite{zelditch} about the generic situation of multiplicity of the eigenvalues of the Laplacian on principal bundles. Besides, we believe that generic properties of the spectrum of warped products are of independent interest.

Our work was motivated by the ideas to split the eigenvalues by using the Hadamard derivative formula of the eigenvalue curves. The use of this technique is common in the works of some authors, see e.g. Albert~\cite{A}, Berger~\cite{Berger}, El Soufi and Ilias~\cite{ahmad}, Henry~\cite{Henry} and Pereira~\cite{pereira}.

Firstly, we establish a kind of splitting eigenvalues theorem considering a family of operators on the base of a warped product, see Eq.~\eqref{Lfmu} and Theorem~\ref{separete_L}. For this class of operators, the parameter is present in the first and zero order coefficients in a very specific way. We point out that the referred class is not covered by the Uhlenbeck's works. As an application of Theorem~\ref{separete_L} we prove a density theorem for a set of warping functions that make the spectrum of the Laplacian a warped--simple spectrum\footnote{See definition given in Section~\ref{Sec-GWSE}.}, cf. Theorem~\ref{warped_theo}.

It is well-known in the literature that results in the spirit of Uhlenbeck~\cite{Uhlenbeck} are no longer true in the presence of symmetries. Thus, one can ask the following question:

\emph{What can be said about the generic situation of multiplicity of the eigenvalues of the Laplacian in the presence of symmetries?}

Arnol'd~\cite{arnold} proposed in an indirect way this question on the context of Dirichlet integral in symmetric domains of $\mathbb{R}^n$. There are also several other works in the literature about this subject, see e.g. \cite{Marrocos, pereira1, pereira2, schueth, zelditch}.

In the Riemannian setting, we highlight two main results. The first one has been recently proven by Schueth~\cite{schueth}, who considered the generic irreducibility of the eigenspace of the Laplacian associated to a left $G$--invariant metric on a Lie group $G$. She translated the problem to Lie algebraic setting and algebraic tools were used instead of the analytic ones. The second one had been before proven by Zelditch~\cite{zelditch}. He treated essentially the case where the group $G$ is finite and assumed that all dimensions of the irreducible representations of $G$ are smaller than the dimension of the manifold. He also emphasized that his technique fail completely without this hypothesis.

We consider a class of $G$--manifolds $N\times G/K$, where the Lie group $G$ (not necessary Abelian) has irreducible representations with arbitrary larger dimensions and it has positive cohomogeneity, i.e., the dimension of the quotient space is nonzero, see Theorem~\ref{generic-G-eigenvalue}. Then, both techniques used by Schueth and Zelditch cannot be applied directly in our case, but even so we apply a Schueth's result to prove the genericity irreducibility of the eigenspace when the smooth $G$--action is free since it is necessary to take a $G$--simple Riemannian metric on $G$, cf. Corollary~\ref{question2}. Besides, Zelditch's work still remains a prototype for us. The real novelty lies in the use of warped product theory to obtain an eigenspace irreducibility result on that class of $G$--manifolds.

\section{Spectrum of Warped Products}

In this section, we fix notation, comments about facts that will be used in our proofs and list without proofs all the main formulas which will be appropriated for us. The material we summarize here is all well-known. Bishop-O'Neill~\cite{oneil}, Ejiri~\cite{ejiri} and Zelditch~\cite{zelditch} are some references for it.

Let $M=B^n\times_f F^m$ be a warped product of two Riemannian manifolds with warped metric $g=g_B+f^2g_F$, where $f$ is a warping function on the base $B$ of $M$. We denote by $\Delta^{B}$, $\Delta^{F}$ and $\Delta^{B\times_fF}$ the Laplacians of $B$, $F$ and $B\times_f F$, respectively.

The next result shows how to compute the Laplacian on $M$. Its proof follows by straightforward computation from \cite[Lemma~7.3]{oneil}  or, alternatively, see  \cite[Lemma~2.1]{ejiri} for a very detailed and readable account.

\begin{lemma} For all $C^2$--function $u$ on $B^n\times_f F^m$ holds
\begin{equation}\label{laplace-warped}
\Delta^{B\times_fF}u= \Delta^B u + \frac{m}{f}g_B(\nabla u, \nabla f) + \frac{1}{f^2}\Delta^F u.
\end{equation}
\end{lemma}

In what follows, it will be fundamental to assume that both the base $B$ and the fiber $F$ of $M$ are compact manifolds. In particular, the spectrum $0=\mu_0< \mu_1\leq\cdots \leq\mu_i\leq\cdots$ of $\Delta^F$ is discrete. Let $\{\psi_i\}_{i=0}^{\infty}$ be a complete basis of $L^2(F)$ given by eigenfunctions of $\Delta^F$. For each eigenvalue $\mu$ such that $\Delta^F\psi=-\mu\psi$ and each $C^2$--function $\phi$ on $B$ we compute
\begin{eqnarray}\label{laplace-warped2}
\nonumber\Delta^{B\times_f F}(\psi\phi) & \stackrel{\eqref{laplace-warped}}{=} & \psi\Delta^B\phi + \psi\frac{m}{f}g_B(\nabla\phi,\nabla f) +\frac{\phi}{f^2}\Delta^F\psi\\
& = & \psi(\Delta^B\phi +\frac{m}{f}g_B(\nabla\phi,\nabla f) -\frac{\mu}{f^2}\phi).
\end{eqnarray}

Let us define a differential operator $L^f_\mu$ on $B$ by
\begin{equation}\label{Lfmu}
L^f_\mu\phi = \Delta^B\phi + \frac{m}{f}g_B(\nabla\phi, \nabla f) -\frac{\mu}{f^2}\phi,
\end{equation}
where $f$ is a positive function on $B$, $\mu$ is a real number and $\phi$ is a function on $B$. From \cite[Lemma~2.3]{ejiri}, $L^f_\mu$ is a strongly elliptic, self-adjoint differential operator of $(B,f^{2m/n}g_B)$.

We now adapt some known facts from perturbation theory to the situation of warped products. Among other things, we deduce an equation for the multiplicity of an eigenvalue of $\Delta^{B\times_f F}$ in terms of $L_\mu^f$, and a variational formula for an eigenvalue of $L_\mu^f$ along a family of warping functions. For a support material, we refer the reader to Bando-Urakawa~\cite{Bando}, Berger~\cite{Berger} and Soufi-Ilias~\cite{ahmad}.

For each eigenvalue $\mu_i$ of $\Delta^F$, we consider $0<\lambda^{\mu_i}_0\leq\lambda^{\mu_i}_1\leq\cdots\leq\lambda^{\mu_i}_j\leq\cdots$ the eigenvalues of the operator $L^f_{\mu_i}$ and $\{\phi_j^{\mu_i}\}_{j=0}^\infty$ a complete basis of $L^2(B,f^{2m/n}g_B)$ such that $L^f_{\mu_i}\phi_j^{\mu_i}=-\lambda_j^{\mu_i}\phi_j^{\mu_i}$. Using Eq.~\eqref{laplace-warped2}, we obtain $\Delta^{B\times_fF}(\psi_{i}\phi_j^{\mu_i})=-\lambda_j^{\mu_i}\psi_{i}\phi_j^{\mu_i}$ which implies that $\{\psi_{i}\phi_j^{\mu_i}\}_{i,j=0}^{\infty}$ is a complete basis of $L^2(B\times_f F)$, cf. \cite[Theorem~2.1]{ejiri}. Moreover, we can deduce that the multiplicity $m_{\Delta^{B\times_fF}}(\lambda)$ of an eigenvalue $\lambda$ of $\Delta^{B\times_f F}$ depends on the multiplicity $m_{L^f_{\mu}}(\lambda)$ of $\lambda$ as an eigenvalue of $L^f_{\mu}$ and the multiplicity $m_{\Delta^F}(\mu)$ of an eigenvalue $\mu$ of $\Delta^F$. More precisely,
\begin{align}\label{multiplicity formula}
m_{\Delta^{B\times_fF}}(\lambda)=\sum_{\mu\in spec(\Delta^F)}m_{L^f_{\mu}}(\lambda)m_{\Delta^F}(\mu).
\end{align}

Let $\mathcal{M}$ denote the Banach space (with the fixed $C^k$~topology) of $C^k$~Riemannian metrics on a compact smooth manifold $M$.

Given an eigenvalue $\lambda_{0}$ of $\Delta_{g_0}$ with multiplicity $m(\lambda_{0})$, there exist a positive number $\epsilon_{\lambda_0, g_0}$ and a neighborhood $\mathcal{V}_{\epsilon}$ in $\mathcal{M}$ such that for all $g\in \mathcal{V}_{\epsilon}$ holds
\begin{equation}\label{kato-continuity}
\sum_{\{|\lambda-\lambda_{0}|<\epsilon_{\lambda_0, g_0}\}\cap spec(\Delta_g)}m(\lambda)=m(\lambda_{0}).
\end{equation}

Let $g(t)$ be a real analytic deformation of $g(0)=g_0$, and let $\Delta_{g(t)}$ be the family of associated Laplace operators. By Kato choice's theorem there are $\epsilon> 0$, real analytic curves $\lambda_k(t)$ and $\varphi_k(t)$ such that
\begin{equation}\label{kato}
\Delta_{g(t)}\varphi_k(t)=-\lambda_k(t)\varphi_k(t) \quad \mbox{and}\quad (\varphi_k(t),\varphi_\ell(t))_t=\delta_{k\ell}\quad\mbox{for all}\quad |t|<\epsilon,
\end{equation}
where $(\cdot,\cdot)_t$ stands for the inner product of $L^2(M,\mathrm{dvol}_{g(t)})$. Moreover, $\{\lambda_k(t)\}_{k=0}^{\infty}$ is the spectrum of $\Delta_{g(t)}$ and $\{\varphi_k(t)\}_{k=0}^{\infty}$ is an orthonormal basis of $L^2(M,dv_{g(t)})$.

We point out that the same results given in \eqref{kato-continuity} and \eqref{kato} hold for the family of the operators $L^f_\mu$, with $f\in C^k_+(B)=\{f\in C^k(B): f>0\}$ endowed with the $C^k$~topology, $k\geq 2$.

We only need to consider linear deformations of the form $g(t)=g_0+tH$, where $H$ is a symmetric covariant $2$--tensor on $M$. We recall that
\begin{equation}\label{deriv-Eigenv}
\lambda'_k=-\int_{M}\langle d\varphi_k\otimes d\varphi_k-\frac{1}{4}\Delta(\varphi_k^2)g_0,H\rangle \dM,
\end{equation}
where $\langle\cdot,\cdot\rangle$ stands for the inner product induced on the space of covariant $2$--tensors on $M$. For a proof see \cite{Berger} or, alternatively, \cite[Proposition~2.1]{ahmad} for a very detailed account.

Now let us consider $M=B^n\times F^m$ endowed with a warped metric $g_0=g_B+ f_0^2g_F$ and a real analytic deformation $g(t)=g_B+f(t)^2g_F$, where $f(t)=f_0+tr$ with $r\in C^k(B)$ and $f_0\in C^k_+(B)$.

Henceforth, since there is no danger of confusion, we will use the same notation $\langle\cdot,\cdot\rangle$ for the inner product on $M$ and $B$ as well as on the space of covariant $2$--tensors on a Riemannian manifold.

Let $\lambda(t)$ be an eigenvalue of $\Delta^{B\times_{f(t)} F}$ and $\psi\phi^\mu$ its corresponding normalized real-valued eigenfunction such that  $(L^f_\mu+\lambda)\phi^\mu=0 $ and $(\Delta^F+\mu)\psi=0$. Since $H=2f_0rg_F$ and $\dM=f^m\mathrm{dvol}_{B\times F}$, we use Fubini's theorem to compute the two terms of Eq.~\eqref{deriv-Eigenv} as follows
\begin{equation*}
\int_{M} \langle d(\psi\phi^\mu)\otimes d(\psi\phi^\mu),2f_0rg_F \rangle \dM=\int_B\frac{2\mu}{f_0^2}(\phi^{\mu})^2f_0^{m+1}r\dB
\end{equation*}
and
\begin{equation*}
\int_{M}\langle-\frac{1}{4}\Delta^{M} (\phi^\mu\psi)^2g,2f_0rg_F\rangle \dM  =  m\int_B\big[\big(\lambda-\frac{\mu}{f_0^2}\big)(\phi^\mu)^2-|\nabla\phi^\mu|^2\big]f_0^{m+1}r\dB.
\end{equation*}
Hence,
\begin{equation}\label{warped_derivative}
\lambda'(r)=-\int_B\big[ \big(m\lambda +(2-m)\frac{\mu}{f_0^2}\big)(\phi^\mu)^2-m|\nabla\phi^\mu|^2\big]f_0^{m+1}r\dB.
\end{equation}

In order to prove our first theorem we will need a variational formula for an eigenvalue $\lambda$ of $L_\mu^f$ along the family $f(t)=f_0+tr$. For this, we note that  $L_\mu^f$ and $\Delta^{B\times_f F}$ share the same eigenvalues, in general with different multiplicity, cf. Eq.~\eqref{multiplicity formula}. Hence, Eq.~\eqref{warped_derivative} is also a variational formula to $\lambda$ as an eigenvalue of $L_\mu^f$. Thus there should be no cause for confusion if we simplify the notation by writing $\lambda$ for $\lambda^\mu$, so that its corresponding eigenfunction is denoted by $\phi$. Notice that
\begin{equation}\label{Prop1-L}
L^{f}_\mu(\phi^2)=2\phi L^{f}_\mu\phi+2|\nabla\phi|^2+\frac{\mu}{f^2}\phi^2
\end{equation}
for all $\phi\in C^2(B).$ With this considerations in mind, we can rewrite \eqref{warped_derivative} as follows.
\begin{lemma}
If $\phi(t)$ and $\lambda(t)$ are differentiable curves of eigenfunctions and eigenvalues of the operators $L_\mu^{f(t)}$, respectively, then
\begin{equation}\label{derivate_L}
\lambda'= \frac{1}{2}\int_B\big[mL^{f_0}_\mu(\phi^2)+(m-4)\frac{\mu}{f_0^2}\phi^2\big]f^{m+1}r\dB.
\end{equation}
\end{lemma}

\section{Genericity of Warped--Simple Eigenvalues}\label{Sec-GWSE}

Here we analyze the generic spectrum of the Laplace operator defined on a compact Riemannian warped product $B\times_f F$ parametrized by warping functions $f$ on $B$. We say that an eigenvalue $\lambda$ of $\Delta^{B\times_fF}$ is \emph{warped--simple} if there is only one eigenvalue $\mu$ of $\Delta^F$ such that $\lambda$ is an eigenvalue of $L^f_{\mu}$ with multiplicity $m_{L^f_{\mu}}(\lambda)=1$, i.e., $m_{\Delta^{B\times_fF}}(\lambda)=m_{\Delta^F}(\mu)$.

We are almost in a position to prove that the set of all warping functions $f$ of class $C^k$ such that $\Delta^{B\times_fF}$ has \emph{warped--simple spectrum} is residual in $C_+^k(B)$. Firstly, as mentioned earlier we establish a splitting eigenvalue theorem for the family of operators $\{L^f_{\mu}\}_{f\in C_+^k(B)}$.
\begin{theorem}\label{separete_L}
Let $\lambda_0$ be an eigenvalue of the operator $L_{\mu}^{f_0}$ with multiplicity $m>1$. Take a positive number $\epsilon_{\lambda_0,f_0}$ and a neighborhood $\mathcal{V}_{\epsilon}$ of $f_0$ in $C_+^k(B)$ as in \eqref{kato-continuity}. Then for each open neighborhood $\mathcal{U}\subset \mathcal{V}_{\epsilon}$ there is $f\in\mathcal{U}$ such that all eigenvalues $\lambda^f$ of $L^f_\mu$ with $|\lambda^f-\lambda_0|<\epsilon_{\lambda_0,f_0}$ are simple.
\end{theorem}
\begin{proof}
The proof is by contradiction. Suppose that there is an open neighborhood $\mathcal{U}\subset \mathcal{V}_\epsilon$ of $f_0$ such that for all $f\in \mathcal{U}$ the eigenvalue $\lambda^f$ of $L_\mu^f$ with $|\lambda^f-\lambda_0|< \epsilon_{\lambda_0, f_0}$ has multiplicity $m>1$. In this case, for any smooth function $r$ the perturbation $f(t)=f_0+tr$ fails to split the eigenvalue $\lambda_0(t)$, and the $m$ curves given by \eqref{kato} must satisfy $\lambda_j'(t)=\lambda_i'(t)$ for small $t$. From Eq.~\eqref{derivate_L} we obtain
\begin{equation}\label{EqAux1Thm3-1}
mL^{f_0}_\mu(\phi_i^2)+(m-4)\frac{\mu}{f_0^2}\phi_i^2=mL^{f_0}_\mu(\phi^2_j)+(m-4)\frac{\mu}{f_0^2}\phi^2_j.
\end{equation}
It follows from \eqref{Prop1-L} that \eqref{EqAux1Thm3-1} is equivalent to
\begin{equation}\label{identity_separete}
|\nabla\phi_i|^2 +\big(-\lambda_0 +\frac{m-2}{m}\frac{\mu}{f_0^2}\big)\phi_i^2=|\nabla\phi_j|^2+\big(-\lambda_0 +\frac{m-2}{m}\frac{\mu}{f_0^2}\big)\phi_j^2.
\end{equation}
Unfortunately, we were not able to use the identity~\eqref{identity_separete} to get a contradiction. So, we assume that it holds in an open neighborhood $\mathcal{U}$ of $f_0$, i.e., for all $f\in\mathcal{U}$ and an eigenvalue $\lambda^f$ with $|\lambda^f-\lambda_0|< \epsilon_{\lambda_0,f_0}$ we have \eqref{identity_separete}.

Let us define for convenience $\xi:=\phi_i-\phi_j$ and $\eta:=\phi_i+\phi_j$. By regrouping some terms in \eqref{identity_separete} we obtain
\begin{equation}\label{identity_separete_2}
\langle\nabla\xi,\nabla\eta\rangle +\big(-\lambda^f +\frac{m-2}{m}\frac{\mu}{f^2}\big)\xi\eta=0.
\end{equation}
Consider $f(t)=f_0+tr$ with $r\in C^k(B)$ such that \eqref{identity_separete_2} holds for all sufficiently small $t$.
Taking the $t$-derivative we get
\begin{equation}\label{identity_separete_2_derivated}
\langle\nabla\dot{\xi},\nabla\eta\rangle + \langle\nabla\xi,\nabla\dot{\eta}\rangle + \big(-\lambda_0 +\frac{m-2}{m}\frac{\mu}{f_0^2}\big)\big(\dot{\xi}\eta + \xi\dot{\eta}\big) -\frac{m-2}{m}\frac{2\mu r}{f_0^3}\xi\eta=\lambda'\xi\eta.
\end{equation}
Since $(L_\mu^{f(t)}+\lambda(t))\xi(t)=0$ and $(L_\mu^{f(t)}+\lambda(t))\eta(t)=0$, we compute
\begin{eqnarray*}
-(L_\mu^{f_0}+\lambda_0)\dot{\xi}=((L_\mu^{f})'+\lambda')\xi \stackrel{\eqref{Lfmu}}{=} m \langle\nabla\frac{r}{f_0},\nabla\xi\rangle + \big(\lambda'+\frac{2\mu r}{f_0^3}\big)\xi,\\
-(L_\mu^{f_0}+\lambda_0)\dot{\eta}=((L_\mu^{f})'+\lambda')\eta \stackrel{\eqref{Lfmu}}{=} m \langle\nabla\frac{r}{f_0},\nabla\eta\rangle + \big(\lambda'+\frac{2\mu r}{f_0^3}\big)\eta.
\end{eqnarray*}
Besides, we note that $\dot{\xi}$ and $\dot{\eta}$ can be taken such that
$$\int_{B}\dot{\xi}\phi f_0^m\dB=\int_{B}\dot{\eta}\phi f_0^m\dB=0$$
for all eigenfunction $\phi$ corresponding to $\lambda_0$.

In order to get more information about eigenfunctions $\xi$ and $\eta$ we look on the left-hand side of \eqref{identity_separete_2_derivated} as an operator $\Gamma$ from $ H^1(B)$ to $L^1(B)$ as follows
\begin{equation}
r\stackrel{\Gamma}{\longmapsto} \langle\nabla\dot{\xi},\nabla\eta\rangle + \langle\nabla\xi,\nabla\dot{\eta}\rangle + \big(-\lambda_0 +\frac{m-2}{m}\frac{\mu}{f_0^2}\big)\big(\dot{\xi}\eta + \xi\dot{\eta}\big) -\frac{m-2}{m}\frac{2\mu r}{f_0^3}\xi\eta
\end{equation}

Now, we can proceed as in Henry \cite[Chapter~7]{Henry} to deduce from \eqref{identity_separete_2_derivated} that $\Gamma$ is a finite rank operator of order $1$, cf. \cite[Section~7.1]{Henry}. Moreover, we can find the principal part of $\Gamma$ by
\begin{equation*}
r\longmapsto kP.V.\int_B \langle G(x,y)\otimes\nabla r(y), Sym(\nabla\xi\otimes\nabla\eta)\rangle\dB,
\end{equation*}
where ``$P.V.$" stands for Cauchy principal value, $G(x,y)$ for the Green's function for the Laplacian, and $Sym$ for the tensor's symmetrization. Since $\Gamma$ has finite rank, it is necessary that
\begin{equation*}
\langle X\otimes Y, Sym(\nabla \xi \otimes \nabla \eta)\rangle=0,
\end{equation*}
cf. \cite[Section~7.4]{Henry}, see also p.~98 of \cite{Henry} to an analogous example. In particular, $X(\xi)X(\eta)=0$ for all $X\in \mathfrak{X}(B)$. As a consequence we have $|\nabla\phi_i|= |\nabla\phi_j|$ on $B$. Coming back to analysis of Eq.~\eqref{identity_separete} we conclude that $\phi_i=\phi_j$ on $B$, which is a contradiction.
\end{proof}

\begin{theorem}\label{warped_theo}
Define $\mathcal{B}$ to be the set of warping functions $f$ of class $C^k$ such that $\Delta^{B\times_fF}$ has warped--simple spectrum. Then $\mathcal{B}$ is residual.
\end{theorem}

\begin{proof}
For each integer $i\geq1$, define the set
\begin{equation*}
B_i=\{f\in C_+^k(B): \mbox{all eigenvalues $\lambda < i$ of $\Delta^{B\times_f F}$ are warped--simple}\}\subset C_+^k(B).
\end{equation*}

Note that $\mathcal{B}=\cap^\infty_{i=1}B_i$. Hence, it is enough to show that for each $i\geq1$, $B_i$ is a dense open set. The openness follows from continuity of the eigenvalues in general metric perturbation setting, see~\cite{Bando} for details.

For the denseness part, we argue by contradiction. Suppose that $B_i$ is not dense for some $i$, then  there are a function $f_0 \in C_+^k(B)$ and an eigenvalue $\lambda< i$ such that for any function $r \in C^k(B)$ the perturbation $g(t)= g_B+(f_0+tr)^2g_F$ fails to split the eigenvalue $\lambda$. In this case, there is at least a pair of curves  $\lambda(t)$ and $\overline{\lambda}(t)$ of eigenvalues of $L_\mu^{f(t)}$ and $L_{\overline{\mu}}^{f(t)}$, respectively, such that $\lambda(0)=\overline{\lambda}(0)=\lambda$ and $\lambda'(t)=\overline{\lambda}'(t)$ for small $t$. Thus, from Eq.~\eqref{warped_derivative} we get
\begin{eqnarray*}
\int_B\big\{ m\lambda[(\phi^\mu)^2-(\phi^{\overline{\mu}})^2] +(2-m)\frac{1}{f_0^2}[\mu(\phi^\mu)^2-\overline{\mu}(\phi^{\overline{\mu}})^2]\\
-m[|\nabla\phi^\mu|^2-|\nabla\phi^{\overline{\mu}}|^2]\big\}f_0^{m+1}r\dB=0.
\end{eqnarray*}
As $r$ is an arbitrary function, we obtain the following equation on the base $B$
\begin{equation}\label{derivative_eq}
\lambda[(\phi^\mu)^2-(\phi^{\overline{\mu}})^2] +\frac{2-m}{mf_0^2}[\mu(\phi^\mu)^2-\overline{\mu}(\phi^{\overline{\mu}})^2]-[|\nabla\phi^\mu|^2-|\nabla\phi^{\overline{\mu}}|^2]=0.
\end{equation}
There are two cases to consider: $\mu=\overline{\mu}$ and $\mu\neq\overline\mu$. The $\mu=\overline{\mu}$ case has been proved in our Theorem~\ref{separete_L}. The idea to show that \eqref{derivative_eq} cannot hold if $\mu\neq\overline{\mu}$ is motivated by an argument of Uhlenbeck~\cite{Uhlenbeck}. We define
\begin{equation*}
\xi:=\phi^\mu+\phi^{\overline{\mu}},\quad \eta:=\phi^\mu-\phi^{\overline{\mu}}\quad \mbox{and}\quad q:=\frac{2-m}{mf_0^2}[\mu(\phi^\mu)^2-\overline{\mu}(\phi^{\overline{\mu}})^2].
\end{equation*}
So, Eq.~\eqref{derivative_eq} becomes $\langle\nabla\eta,\nabla\xi\rangle-\lambda\xi\eta=q$.

Now, it is enough to consider $m\neq2$, as the $m=2$ case follows from Uhlenbeck's argument. For this, we take an integral curve $x(t)$ of $\nabla\xi$, so that $u(t)=\eta(x(t))$ is a solution to $\frac{du}{dt}-\lambda\xi u=q$. Without loss of generality, we can choose $x_0$ in $B$ such that $\xi(x_0)>0$. So, $\xi(x(t))$ is a non-decreasing bounded function on $B$, and the solution of the latter ordinary differential equation must satisfy: either $u(t)$ is unbounded or $\lim_{t\to\infty}u(t)=0$. Then only the second occurs, since $u(t)=\eta(x(t))$ is bounded.

Let $x_\omega$ be a point in \emph{$\omega$--limit set} of $x_0$, then there is a sequence $x(t_k)$ such that $\lim_{k\to\infty} x(t_k)=x_\omega$ and $\nabla\xi(x_\omega)=0$. Hence, we have
\begin{equation*}
0=\lim_{k\to\infty}u(t_k)=\eta(x_\omega)=(\phi^\mu-\phi^{\overline{\mu}})(x_\omega)
\end{equation*}
and
\begin{equation*}
q(x_\omega)=\lim_{k\to\infty} q(x(t_k))=\lim_{k\to\infty} \big(\langle\nabla\eta (x(t_k)),\nabla\xi (x(t_k))\rangle-\lambda\xi (x(t_k))\eta(x(t_k)\big)=0.
\end{equation*}
Whence, we conclude that
\begin{equation*}
0=q(x_\omega)=\frac{2-m}{m(f_0(x_\omega))^2}(\mu-\overline{\mu})(\phi^\mu(x_\omega))^2,
\end{equation*}
which is a contradiction, since $m\neq2$, $\mu\neq\overline{\mu}$ and $2\phi^\mu(x_\omega)=\xi(x_\omega)> 0$. This completes our proof.
\end{proof}

\section{Riemannian $G$--Manifolds}

In this section, we apply Theorem~\ref{warped_theo} to give an answer to the generic situation of multiplicity of the eigenvalues of the Laplacian on a class of compact $G$--manifolds, where $G$ is a compact Lie group. We recall that a \emph{$G$--manifold} is a smooth manifold $M$ endowed with a smooth effective action by a Lie group $G$. If for a Riemannian metric $g$ on $M$ holds $G\subset Iso(M,g)$ then we will call it $G$--invariant metric and $(M,g)$ a Riemannian $G$--manifold. The space of all $G$--invariant $C^k$~Riemannian metrics on $M$ will be denoted by $\mathcal{M}^k_G$.

In a Riemannian $G$--manifold $(M,g)$ each Laplace eigenspace $E_\lambda$ is a linear real representation of $G$, since the Laplacian commutes with isometries. In general, $E_\lambda$ is reducible,  in this case, it is possible to split them into irreducible real representations of $G$, namely $E_\lambda=\bigoplus_{\sigma}m(\lambda,\sigma,g)W_\sigma$, where $(\sigma, W_\sigma)$ runs over the set of all equivalence class of irreducible representations of $G$ and $m(\lambda,\sigma,g)$ is the multiplicity of $(\sigma, W_\sigma)$ in $E_\lambda$.

We say that an eigenvalue $\lambda$ of $\Delta_g$ is \emph{$G$--simple} if its corresponding eigenspace is an irreducible real representation of $G$. When all eigenvalues are $G$--simple we denote the spectrum of $\Delta_g$ as \emph{$G$--simple spectrum} and the associated Riemannian metric $g$ as \emph{$G$--simple metric}.

Let $M=B\times G/K$ be a $G$--manifold, where $B$ is a compact smooth manifold and $G/K$ is a compact homogeneous space. Suppose that each element $\tilde\gamma$ of $G$ (as a diffeomorphism of $M$) is given by $\tilde\gamma(x,y)= (x,\gamma y)$, with $\gamma$ acting on $G/K$. Our main result is

\begin{theorem}\label{G-eigenvalue-separate}
Let $M=B\times G/K$ be a $G$--manifold where $G$ is given as above, and let $g_{G/K}$ be a $G$--simple metric on $G/K$. Given a metric $g_0$ in $\mathcal{M}^k_G$ and an eigenvalue $\lambda_0$ of $\Delta_{g_0}$ such that its corresponding eigenspace is not irreducible representation of $G$. Consider a positive number $\epsilon_{\lambda_0,g_0}$ and a neighborhood $\mathcal{V}_{\epsilon}$ of $g_0$ in $\mathcal{M}^k_G$ as in \eqref{kato-continuity}. Then for each open neighborhood $\mathcal{U}\subset \mathcal{V}_{\epsilon}$ there is $g\in\mathcal{U}$ such that all eigenvalues $\lambda$ of $\Delta_g$ with $|\lambda-\lambda_0|<\epsilon_{\lambda_0,g_0}$ are $G$--simple.
\end{theorem}
\begin{proof}
Let $g_{G/K}$ be a $G$--simple metric on $G/K$, and let $g_B$ be a Riemannian metric on $B$. By Theorem~\ref{warped_theo} we can choose a function $f\in C^k_+(B)$ such that $g_W:=g_B+f^2g_{G/K}$ is a warped--simple metric on $M=B\times_f G/K$. Note that  $g_W$ is a $G$--simple metric, since $g_{G/K}$ is a $G$--simple metric on $G/K$.

Consider $\lambda_0$ an eigenvalue of $\Delta_{g_0}$ (on $M$) with corresponding eigenspace $E_{\lambda_0}=W_1\oplus W_2\oplus\cdots\oplus W_r$, $r>1$, splitting into irreducible representations of $G$ (not necessarily non-isomorphic each other). In this situation \eqref{kato} may be adapted to $G$--invariant metric perturbations case to conclude that there is at most $r$ different analytic curves of eigenvalues, cf. \cite[Proposition~1.28]{zelditch}. Taking the analytic family of Riemannian metrics $g(t)=(1-t)g_0+tg_W$, those $r$ curves may cross each other just finite number of times, since  the eigenvalues of $g(1)=g_W$ are all $G$--simple. Then, we can choose $t^*\in (0,1)$ such that all eigenvalues $\lambda(t^*)$, $|\lambda(t^*)-\lambda_0|<\epsilon_{\lambda_0,g_0}$ are $G$--simple eigenvalues, hence the Riemannian metric $g(t^*)$ has the property desired.
\end{proof}

\begin{corollary}\label{generic-G-eigenvalue}
Under the same hypothesis as in Theorem~\ref{G-eigenvalue-separate} the set of $G$--simple metrics in $\mathcal{M}^k_G$ is residual.
\end{corollary}
\begin{proof}
The argument follows the same lines of the proof of Theorem~\ref{warped_theo} since the eigenvalues of the Laplacian are continuous as functions on $\mathcal{M}^k_G$ and we can split them using Theorem~\ref{G-eigenvalue-separate}.
\end{proof}

A prototype example of symmetric space $G/K$ equipped with a $G$--simple metric is $SO(n+1)/SO(n)$ equipped with the canonical metric. However, such metric always exists in the case of symmetric spaces given by quotients of Lie groups $G$ modulo maximal compact subgroups $K$, see \cite[Theorem~2 and examples of p.~249]{gurarie}.

Schueth~\cite{schueth} established an algebraic criteria for the existence of a $G$-left invariant metric $g$ on $G$ such that all Laplace eigenspaces are irreducible representations of $G$. In particular, it holds for $SO(3), SO(4), U(4)$ and $SU(2)\times SU(2)\times\cdots\times SU(2)\times T$, where $T$ is a torus. We can put Schueth criteria and the Corollary~\ref{generic-G-eigenvalue} together in order to obtain the generic result for trivial principal $G$--bundle. Obviously, the result holds for any principal $G$--bundle $P$ isomorphic to $M\times G$.

\begin{corollary}\label{question2}
Let $P$ be a principal $G$--bundle isomorphic to the trivial principal $G$--bundle. If $G$ satisfies the Schueth criteria then the set of $G$--simple metrics on $P$ is residual in $\mathcal{M}^k_G$.
\end{corollary}

The irreducibility of the eigenspaces of the Laplacian on general principal $G$--bundle is an open problem yet. We will hope to address that question in the next work.

\noindent\textbf{Acknowledgements:}
The authors would like to express their sincere thanks to Department of Mathematics at Lehigh University, where part of this work was carried out. They are grateful to Huai-Dong Cao and Mary Ann for their warm hospitality and constant encouragement. The first author is also grateful to G\'erard Besson for his time inspiring and helpful discussions, as well as to Institut Fourier in Grenoble France for a good atmosphere while this work was done.

\end{document}